\documentclass[11pt,reqno]{amsart}
\usepackage{amsmath,amsthm,amscd,amsfonts,amssymb,color}
\usepackage{cite}
\usepackage[mathscr]{eucal}

\usepackage[bookmarksnumbered,colorlinks,plainpages]{hyperref}
\newtheorem{theorem}{Theorem}[section]

\newtheorem*{Acknowledgement}{\textnormal{\textbf{Acknowledgement}}}

\newtheorem{corollary}[theorem]{Corollary}
\theoremstyle{definition}
\newtheorem{definition}[theorem]{Definition}
\newtheorem{example}[theorem]{Example}
\newtheorem{Open Prob}[theorem]{Open Problem}
\theoremstyle{remark}

\numberwithin{equation}{section}
\def\DJ{\leavevmode\setbox0=\hbox{D}\kern0pt\rlap{\kern.04em\raise.188\ht0\hbox{-}}D}
\begin{document}

\title[Best proximity results for $p$-proximal contractions]{Best proximity results for $p$-proximal contractions on topological spaces}

\author[ A.\ Bera,  L.K.\ Dey, A. \ Petrusel, A.\ Chanda]
{ Ashis Bera$^{1}$, Lakshmi Kanta Dey$^{2}$, Adrian Petrusel$^{3}$, Ankush Chanda$^{4}$} 

\address{{$^{1}$} Ashis Bera
                    Department of Mathematics,
                    National Institute of Technology
                    Durgapur, India.}
                    \email{beraashis.math@gmail.com}
                   
\address{{$^{2}$} Lakshmi Kanta Dey,
                    Department of Mathematics,
                    National Institute of Technology
                    Durgapur, India.}
                    \email{lakshmikdey@yahoo.co.in}
    \address{{$^{3}$} Adrian Petrusel,
                    Department of Mathematics, Babes-Bolyai University Cluj-Napoca and          Academy of Romanian Scientists Bucharest, Romania.}
\email{petrusel@math.ubbcluj.ro}
                    
\address{{$^{4}$} Ankush Chanda,
                    Department of Mathematics,
                    Vellore Institute of Technology
                    Vellore, India.}
                  \email{ankushchanda8@gmail.com}

\keywords{Best proximity points, topological spaces, $p$-proximal contractions. \\
\indent 2020 {\it Mathematics Subject Classification}. $47$H$09$, $47$H$10$.}

\begin{abstract}
In this article, we investigate for some sufficient conditions for the existence and uniqueness of best proximity points for the topological $p$-proximal contractions and $p$-proximal contractive mappings on arbitrary topological spaces. Moreover, our results are  authenticated by a few numerical examples and these generalize some of the known results in the literature.
\end{abstract}

\maketitle

\setcounter{page}{1}



\section{\bf Introduction and preliminaries}
\baselineskip .55 cm
The metric fixed point theory serves as a consequential tool to get to the bottom of a huge number of applications in many scientific fields, to name a few, mathematics, computer science and economics. One can note that numerous problems related to preceding domains of research can be expressed as equations of the form $Tx= x$, where $T$ is a self-mapping defined in a suitable underlying structure. Whenever $T$ is not a self-mapping, the equation $Tx=x$ does not necessarily possess a solution. In these circumstances, it is worthy to search for an optimum approximate solution which minimizes the error due to approximation. In other words, for a non-self-mapping $T:A\to B$ defined on a metric space, one enquires for an approximate solution $x^*$ in $A$ so that the error $d(x^*, Tx^*)$ is minimum. This solution $x^*$ is said to be a best proximity point of the mapping. The concerning theory dealing with existence of best proximity points is rich enough and for some exciting findings keen readers are referred to \cite{V,KDKC,EV,WPB,SSJ,DD,DM} and references therein. 

Very recently, Raj and Piramatchi \cite{RP} extended the notion of best proximity points from usual metric spaces to arbitrary topological spaces and affirmed certain related results alongside an generalized version of the well-known Edelstein fixed point theorem. Following this direction of research, later on Som et al. \cite{SLDH} proposed couple of new concepts namely topologically Berinde weak proximal contractions and topologically proximal weakly contractive mappings with respect to a real-valued continuous function $g$ defined on $X\times X$ and investigated for sufficient conditions for the existence and uniqueness of best proximity points for the preceding class of mappings.

On the other hand, of late, Altun \cite{AAS} coined the notion of $p$-proximal contractions and $p$-proximal contractive non-self-mappings  on metric spaces considering the concepts of $p$-contractions and $p$-contractive mappings. Firstly, we recollect the definition of a $p$-proximal contraction.
\begin{definition}
Let $T:A\to B$ be a mapping defined on two non-empty subsets $A$ and $B$ of a metric space $(X,d)$. Then $T$ is said to be a $p$-proximal
contraction if there exists $k\in (0,1)$ such that
\[d(u_1, Tx_1)=dist(A,B),~~ d(u_2, Tx_2)=dist(A,B)\]
imply
\[d(u_1,u_2)\leq k \left( d(x_1, x_2)+|d(u_1, x_1)-d(u_2, x_2)|\right) ,\]
for all $x_1,x_2,u_1,u_2 \in A.$
\end{definition}
As a particular case, whenever $A=B=X$, then the previous definition reduces to the collection of self-mappings known as $p$-contractions introduced by Popescu \cite{P} beforehand. In the following, we recall some definitions and notations which are playing crucial roles in this article. For a detailed reading, one is referred to \cite{SLDH}
\begin{definition}
Let $\Xi$ be a topological space and $\Phi:\Xi\times\Xi\to \mathbb{R}$ be a continuous mapping. Let $(\xi_n)$ be a sequence in $\Xi$. Then $(\xi_n)$ is said to be $\Phi$-convergent to $\xi \in \Xi$ if $|\Phi(\xi_n, \xi)|\to 0$ as $n\to \infty$.
\end{definition}
\begin{definition}
Let $\Xi$ be a topological space and $\Phi:\Xi\times\Xi\to \mathbb{R}$ be a continuous mapping. Let $(\xi_n)$ be a sequence in $\Xi$. Then $(\xi_n)$ is said to be $\Phi$-Cauchy to if $|\Phi(\xi_n, \xi_{n+p})|\to 0$ as $n\to \infty~~\mbox{and}~~p=1,2,3\dots$
\end{definition}
\begin{definition}
Let $\Xi$ be a topological space and $\Phi:\Xi\times\Xi\to \mathbb{R}$ be a continuous mapping. Then $\Xi$ is said to be $\Phi$-complete if every $\Phi$-Cauchy sequence $(\xi_n)$ in $\Xi$ is $\Phi$-convergent to an element in $\Xi$.
\end{definition} 
\begin{definition}  \cite{RP}
Let $\Re, \Omega$ be two non-empty subsets of a topological space $\Xi$ and $\Phi:\Xi\times\Xi\to \mathbb{R}$ be a continuous mapping. Consider $$D_\Phi(\Re, \Omega)=\inf\left\{|\Phi(\alpha, \beta)|:\alpha\in\Re, \beta\in\Omega\right\}.$$
\end{definition}
It is clear that if $\Xi$ is a metric space and $\Phi$ is a distance function, then $D_\Phi(\Re, \Omega)$ is the distance between the two aforementioned sets. Here we consider the following two sets which will be important in establishing our findings.
$$\Re_\Phi=\left\{\alpha\in\Re:|\Phi(\alpha, \beta)|=D_\Phi(\Re, \Omega)~~\mbox{ for some}~~\beta\in\Omega\right\},$$
$$\Omega_\Phi=\left\{\alpha\in\Omega:|\Phi(\alpha, \beta)|=D_\Phi(\Re, \Omega)~~\mbox{ for some}~~\beta\in\Re\right\}.$$
\begin{definition} \cite{RP}
Let $\Re, \Omega$ be two non-empty subsets of a topological space $\Xi$ with $\Re_\Phi\neq \phi$ and $\Phi:\Xi\times\Xi\to \mathbb{R}$ be a continuous mapping. Then $(\Re, \Omega)$ is said to satisfy topologically $p$-property if
\begin{equation}
 \left.
        \begin{array}{ll}
            |\Phi(\alpha_1, \beta_1)|=D_\Phi(\Re, \Omega)\\
            |\Phi(\alpha_2, \beta_2)|=D_\Phi(\Re, \Omega)
        \end{array}
    \right\}\Rightarrow |\Phi(\alpha_1, \alpha_2)|=|\Phi(\beta_1, \beta_2)|
\end{equation}
for $\alpha_1, \alpha_2\in \Re$ and $\beta_1, \beta_2\in \Omega.$
\end{definition}
One can note that if $\Xi$ is a metric space and $\Phi$ is a distance function on $\Xi$, then it is the usual $p$-property, which is the fundamental property to obtain the best proximity points for non-self mappings, introduced in \cite{R}. In our next example, we show that a pair of subsets $(\Re, \Omega)$ of a topological space $\Xi$ may satisfy $p$-property with respect to one mapping but may not so for another mapping.
\begin{example}
Consider $\mathbb{R}^2$ with usual topology. Let $\Re=\{2\}\times [-2,0]$ and $\Omega=\{2\}\times [0, 2]$. Let $\mathcal{F}_1:\mathbb{R}^2\times\mathbb{R}^2\to \mathbb{R}$ and $\mathcal{F}_2:\mathbb{R}^2\times\mathbb{R}^2\to \mathbb{R}$ be defined as
\begin{align*}
\mathcal{F}_1\left((\alpha_1, \alpha_2), (\beta_1, \beta_2)\right)=\alpha_2^2-\beta_2^2~~ \mbox{ for all}~~(\alpha_1, \alpha_2), (\beta_1, \beta_2)\in \mathbb{R}^2
\end{align*}
and
\begin{align*}
\mathcal{F}_2\left((\alpha_1, \alpha_2), (\beta_1, \beta_2)\right)=\alpha_2\beta_2~~ \mbox{ for all}~~(\alpha_1, \alpha_2), (\beta_1, \beta_2)\in \mathbb{R}^2
\end{align*}
respectively. Then $\mathcal{F}_1$ is a continuous function and also $D_{\mathcal{F}_1}(\Re, \Omega)=0$. Now it can be easily verified that the  pair $(\Re, \Omega)$ of subsets of a topological space $\Xi$ satisfies $p$-property with respect to the mapping $\mathcal{F}_1$ but does not so for the mapping $\mathcal{F}_2$ though $\mathcal{F}_2$ is continuous function and $D_{\mathcal{F}_2}(\Re, \Omega)=0$. As $\mathcal{F}_2\left((1,-\frac{1}{2}),(1, 0)\right)=0=D_{\mathcal{F}_2}(\Re, \Omega)$ and $\mathcal{F}_2\left((1,-\frac{1}{3}),(1, 0)\right)=0=D_{\mathcal{F}_2}(\Re, \Omega)$ but $\mathcal{F}_2\left((1,-\frac{1}{2}),(1, -\frac{1}{3})\right)\neq \mathcal{F}_2\left((1,0),(1, 0)\right). $
\end{example} 
In the following, we extend the notions of proximal contractions and modified proximal contractions \cite{AAS} in the context of a topological space.
\begin{definition}
Let $\Re, \Omega$ be two non-empty subsets of a topological space $\Xi$ and $\Phi:\Xi\times\Xi\to \mathbb{R}$ be a continuous mapping. A mapping $\mathcal{F}$ is said to be a topologically proximal contraction with respect to $\Phi$ if there exists a real number $c\in(0, 1)$ such that
\begin{equation}
\left.
        \begin{array}{ll}
            |\Phi(\alpha_1,\mathcal{F}( \beta_1))|=D_\Phi(\Re, \Omega)\\
            |\Phi(\alpha_2, \mathcal{F}(\beta_2))|=D_\Phi(\Re, \Omega)
        \end{array}
    \right\}\Rightarrow |\Phi(\alpha_1, \alpha_2)|\leq c|\Phi(\beta_1, \beta_2)|
\end{equation}
for all $\alpha_1, \alpha_2,\beta_1, \beta_2\in \Re.$
\end{definition}
Now for distinct $\beta_1$ and $\beta_2$, we give a modified definition of topological proximal contraction, which as follows:
\begin{definition}
Let $\Re, \Omega$ be two non-empty subsets of a topological space $\Xi$ and $\Phi:\Xi\times\Xi\to \mathbb{R}$ be a continuous mapping. A mapping $\mathcal{F}$ is said to be a modified topologically proximal contraction with respect to $\Phi$ if there exists a real number $c\in(0, 1)$ such that
\begin{equation}\label{equa1}
\left.
        \begin{array}{ll}
            |\Phi(\alpha_1,\mathcal{F}( \beta_1))|=D_\Phi(\Re, \Omega)\\
            |\Phi(\alpha_2, \mathcal{F}(\beta_2))|=D_\Phi(\Re, \Omega)
        \end{array}
    \right\}\Rightarrow |\Phi(\alpha_1, \alpha_2)|\leq c|\Phi(\beta_1, \beta_2)|
\end{equation}
for all $\alpha_1, \alpha_2,\beta_1, \beta_2\in \Re$ with $\beta_1\neq\beta_2.$
\end{definition}
In our next example, we show that every topological proximal contraction is a modified topological proximal contraction but the converse may not hold. 
\begin{example}
Consider $\mathbb{R}^2$ with usual topology. Let $\Re=\{(0,-\frac{1}{2}),(0, \frac{1}{2})\}$ and $\Omega=\{(0,0), (1,\frac{3}{4}), (1, 5)\}$. Let $\Phi:\mathbb{R}^2\times\mathbb{R}^2\to \mathbb{R}$ be a continuous mapping defined as
\begin{align*}
\Phi\left((\alpha_1, \alpha_2), (\beta_1, \beta_2)\right)=(\alpha_1-\beta_1)+(\alpha_2-\beta_2)~~ \mbox{ for all}~~(\alpha_1, \alpha_2), (\beta_1, \beta_2)\in \mathbb{R}^2.
\end{align*}
Then $D_{\Phi}(\Re, \Omega)=\frac{1}{2}$ and define the mapping $\mathcal{F}$ as follows: 
\[\mathcal{F}(x,y)= \left\{\begin{array}{ll}
\left( 1, 1+\frac{y}{2} \right), & \mbox{if $-1\leq y<0$};\\
\left(0, 1-2y\right), & \mbox{if $0\leq y\leq 1$.}\\
\end{array} \right. \]
Then for $\beta_1=\beta_2=(0,\frac{1}{2}) $ and $\alpha_1=(0, \frac{1}{2}), \alpha_2=(0, -\frac{1}{2})$, we get
\begin{equation*}
        \begin{array}{ll}
            |\Phi(\alpha_1,\mathcal{F}( \beta_1))|=D_\Phi(\Re, \Omega)=\frac{1}{2}\\
            |\Phi(\alpha_2, \mathcal{F}(\beta_2))|=D_\Phi(\Re, \Omega)=\frac{1}{2}
        \end{array}
\end{equation*}
but $|\Phi(\alpha_1, \alpha_2)|=1>c|\Phi(\beta_1, \beta_2)|$ for any $c\in (0,1).$ Therefore $\mathcal{F}$ is not a topological proximal contraction. However, we can not find two distinct $\beta_1$ and $\beta_2$ such that $|\Phi(\alpha_1,\mathcal{F}( \beta_1))|=D_\Phi(\Re, \Omega)$ and $ |\Phi(\alpha_2, \mathcal{F}(\beta_2))|=D_\Phi(\Re, \Omega)$ hold but the implication \eqref{equa1} holds for all $\alpha_1,\alpha_2\in\Re.$ Therefore $\mathcal{F}$ is a modified topological proximal contraction.
\end{example} 
\section{\bf Best proximity point theorem on topological spaces}
To begin with, we introduce the notion of the topological $p$-proximal contractions.
\begin{definition}\label{def1}
Let $\Re, \Omega$ be two non-empty subsets of a topological space $\Xi$. A mapping $\mathcal{F}:\Re\to\Omega$ is  said to be a topological $p$-proximal contraction if there exists a real number $c\in (0,1)$ such that
\begin{align}\label{equa2}
\left.
        \begin{array}{ll}
            |\Phi(\alpha_1,\mathcal{F}( \beta_1))|=D_\Phi(\Re, \Omega)\\
            |\Phi(\alpha_2, \mathcal{F}(\beta_2))|=D_\Phi(\Re, \Omega)
        \end{array}
    \right\}\Rightarrow |\Phi(\alpha_1, \alpha_2)|\leq c\left(|\Phi(\beta_1, \beta_2)|+||\Phi(\alpha_1, \beta_1)|-|\Phi(\alpha_2,\beta_2)||\right)
\end{align}
for all $\alpha_1, \alpha_2,\beta_1, \beta_2\in \Re$ with $\beta_1\neq\beta_2$.
\end{definition}
Note that if $\Xi$ is a metric space and $\Phi$ is a distance function on $\Xi$, then $\mathcal{F}$ becomes the usual $p$-proximal contraction, introduced in \cite{AAS}. Additionally, if we consider $\Re=\Omega=\Xi$, then \eqref{equa2} can be written as
\begin{equation*}
\Phi(\mathcal{F}(\beta_1),\mathcal{F}(\beta_2))\leq c\left(|\Phi(\beta_1,\beta_2)|+|\Phi(\beta_1, \mathcal{F}(\beta_2))-\Phi(\beta_2,\mathcal{F}(\beta_2))|\right)
\end{equation*}
for all $\beta_1, \beta_2\in \Xi$. In such cases, the mapping  satisfying the previous contraction is called a $p$-contraction. The following example shows that a mapping $\mathcal{F}:\Re\to\Omega$, where $\Re, \Omega$ are non-empty subsets of a topological space $\Xi$, may be a topological $p$-proximal contraction with respect to a mapping $\Phi:\Xi\times \Xi\to \mathbb{R}$ and may not be such with respect to another mapping $g:\Xi\times \Xi\to \mathbb{R}$.
\begin{example}
Consider $\mathbb{R}^2$ with usual topology. Let $\Re=\left\{\left(0,\frac{t}{2}\right):0\leq t\leq 1\right\}$ and $\Omega=\left\{\left(\frac{2}{3},\frac{s}{2}\right):0\leq s\leq 1\right\}$. We define the mapping $\mathcal{F}:\Re\to\Omega$ as

\[\mathcal{F}(x,y)=\left\{\begin{array}{ll}
\left(1,\frac{y-1}{4}\right), \mbox{if $x=0$},\\
\left(1,\frac{y}{2}\right), \mbox{if $x=\frac{2}{3}$}.\\
\end{array}\right.
\]
Let $\Phi:\Xi\times\Xi\to \mathbb{R}$ defined by
$$\Phi\left((\alpha_1,\alpha_2),(\beta_1,\beta_2)\right)=\alpha_2-\beta_2.$$ Then $\Phi$ is a continuous mapping and $D_\Phi(\Re, \Omega)=0.$ Next we show that $\mathcal{F}$ is a topological $p$-proximal contraction with respect to $\Phi$. Let $\alpha_1=(0,a_1),\alpha_2=(0, a_2), \beta_1=(0, b_1), \beta_2=(0, b_2)\in\Re$ and 
\begin{equation*}
        \begin{array}{ll}
            |\Phi(\alpha_1,\mathcal{F}( \beta_1))|=D_\Phi(\Re, \Omega)\\
            |\Phi(\alpha_2, \mathcal{F}(\beta_2))|=D_\Phi(\Re, \Omega)
        \end{array}
\end{equation*}
gives that
\begin{align*}
\left|\Phi\left((0, a_1), \left(1,\frac{b_1-1}{4}\right)\right)\right|=0\\
\Rightarrow \left| a_1-\frac{b_1-1}{4}\right|=0.
\end{align*}
Similarly, we get $\left| a_2-\frac{b_2-1}{4}\right|=0.$
Then
\begin{align*}
\left|\Phi(\alpha_1,\alpha_2)\right|&=\left|\frac{b_1-1}{4}-\frac{b_2-1}{4}\right|\\
&\leq\frac{2}{3}|b_1-b_2|+\frac{2}{3}\left\|a_1-b_1|-|a_2-b_2|\right|\\
&= c\left(|\Phi(\beta_1, \beta_2)|+||\Phi(\alpha_1, \beta_1)|-|\Phi(\alpha_2,\beta_2)||\right)
\end{align*}
where $c=\frac{2}{3}$. This shows that $\mathcal{F}$ is a topological $p$-proximal contraction with respect to $\Phi$. On the other hand, let $g:\mathbb{R}^2\times\mathbb{R}^2\to \mathbb{R}$ defined by $g\left((\alpha_1,\alpha_2),(\beta_1,\beta_2)\right)=\alpha_2+\beta_2.$ It can be seen that $D_g(\Re, \Omega)=0$. Let $\alpha_1=(0,\frac{1}{2}),\alpha_2=(0,0), \beta_1=(0,0),\beta_2=(0,\frac{1}{4})\in \Re$, the we get $\left|g\left(\alpha_1,\mathcal{F}(\beta_1)\right)\right|=0$ and $\left|g\left(\alpha_2,\mathcal{F}(\beta_2)\right)\right|=0$. Therefore it is easy to verify that there does not exist any $c\in(0,1)$ such that 
$$|g(\alpha_1, \alpha_2)|\leq c\left(|\Phi(\beta_1, \beta_2)|+||\Phi(\alpha_1, \beta_1)|-|\Phi(\alpha_2,\beta_2)||\right)$$ for all $\alpha_1, \alpha_2,\beta_1,\beta_2$ with $\beta_1\neq\beta_2$. This shows that $\mathcal{F}$ is not a topological $p$-proximal contraction with respect to $g$.
\end{example}
In the subsequent example, we show that there exists a topological space $\Xi$ and a mapping $\mathcal{F}:\Re\to \Omega$, where $\Re, \Omega$ are non-empty subsets of $\Xi$, such that $\mathcal{F}$ is a topological $p$-proximal contraction with respect to a continuous real-valued function $\Phi$. But if the topological space is metrizable with respect to a metric $d$, then the mapping $\mathcal{F}$ is not a $p$-proximal contraction with respect to the metric $d$.
\begin{example}
Consider $\mathbb{R}$ with the usual topology. Let $\Phi:\mathbb{R}\times\mathbb{R}\to \mathbb{R}$ be defined by
$$\Phi(u, v)=u^2-v^2,~~\mbox{for all}~ u, v\in \mathbb{R}.$$ Then $\Phi$ is a continuous function. Let $\Re=\{t:0\leq t\leq 1\}$ and $\Omega=\{s:0\leq s\leq 2\}$ and $\mathcal{F}:\Re\to\Omega$ be defined as $\mathcal{F}(x)=\frac{x}{2},$ $x\in \Re$. Then it can be easily verified that $D_\Phi(\Re,\Omega)=0$. Let $\alpha_1=\frac{1}{2},\alpha_2=\frac{1}{4},\beta_1=1,\beta_2=\frac{1}{4}$. Therefore,
\begin{equation*}
        \begin{array}{ll}
            |\Phi(\alpha_1,\mathcal{F}( \beta_1))|=D_\Phi(\Re, \Omega)=0\\
            |\Phi(\alpha_2, \mathcal{F}(\beta_2))|=D_\Phi(\Re, \Omega)=0.
        \end{array}
\end{equation*}
Then it can be verified that for $c=\frac{1}{4}$,
$$\frac{3}{16}=|g(\alpha_1, \alpha_2)|\leq c\left(|\Phi(\beta_1, \beta_2)|+||\Phi(\alpha_1, \beta_1)|-|\Phi(\alpha_2,\beta_2)||\right)$$ for all $\alpha_1, \alpha_2,\beta_1,\beta_2$ with $\beta_1\neq\beta_2$. Thus $\mathcal{F}$ is a topological $p$-proximal contraction with respect to $\Phi$. On the other hand, let $d$ be the usual metric on $\mathbb{R}$ and $D(\Re,\Omega)=\inf\{d(u,v):u\in\Re,v\in \Omega\}$. Therefore $D_d(\Re,\Omega)=0$. Hence  it can be easily verified that $\mathcal{F}$ is not a $p$-proximal contraction with respect to $d$ for $c=\frac{1}{4}$.
\end{example} 
Now we are in a position to present a best proximity point theorem concerning the newly introduced topological $p$-proximal contractions on an arbitrary topological space.
\begin{theorem}\label{thm1}
Let $\Xi$ be a $\Phi$-complete topological space, where $\Phi:\Xi\times\Xi\to \mathbb{R}$ is a continuous function such that $\Phi(x,y)=0\Leftrightarrow x=y,$ $\|\Phi(x,y)\|=\|\Phi(y,x)\|$ and $\|\Phi(x,z)\|\leq \|\Phi(x,y)\|+\|\Phi(y,z)\| $ for all $x,y,z\in \Xi$. Let $\Re,\Omega$ be non-empty, $\Phi$-closed subsets of $\Xi$ such that $\Omega$ is approximately $\Phi$-compact with respect to $\Re$. Let $\mathcal{F}:\Re\to\Omega$ be a topological $p$-proximal contraction with respect to $\Phi$ such that $\mathcal{F}(\Re_\Phi)\subseteq \Omega_\Phi$ and $\Re_\Phi$ is non-empty. Then $\mathcal{F}$ has a unique best proximity point.
\end{theorem}
\begin{proof}
Let $\xi_0\in\Re_\Phi$ be an arbitrary element. Since $\mathcal{F}(\xi_0)\in\mathcal{F}(\Re_\Phi)\subseteq\Omega_\Phi$, there exists $\xi_1\in \Re_\Phi$ such that $|\Phi(\xi_1, \mathcal{F}(\xi_0))|=D_\Phi(\Re,\Omega)$. Similarly as $\mathcal{F}(\mathcal{F}(\xi_1))\in\mathcal{F}(\Re_\Phi)\subseteq\Omega_\Phi$, there exists $\xi_2\in \Re_\Phi$ such that $|\Phi(\xi_2, \mathcal{F}(\xi_1))|=D_\Phi(\Re,\Omega)$. Proceeding in this way, we can construct a sequence $(\xi_n)$ in $\Re_\Phi$ such that
\begin{equation}\label{equa3}
|\Phi(\xi_{n+1}, \mathcal{F}(\xi_n))|=D_\Phi(\Re,\Omega)~~\mbox{for all}~~n\in\mathbb{N}.
\end{equation}
Now, if there exists an $n\in\mathbb{N}$ such that $\xi_n=\xi_{n+1}$, then it is clear that $\xi_n$ is the best proximity point of the mapping $\mathcal{F}$. Let we assume that $\xi_n\neq\xi_{n+1}$ for all $n\in \mathbb{N}$. From the construction of the sequence, we have
\begin{equation*}
\begin{array}{ll}
|\Phi(\xi_n,\mathcal{F}( \xi_{n-1}))|=D_\Phi(\Re, \Omega)\\
|\Phi(\xi_{n+1}, \mathcal{F}(\xi_n))|=D_\Phi(\Re, \Omega),
\end{array}
\end{equation*}
for all $n\in\mathbb{N}$. As $\mathcal{F}$ is a topological $p$-proximal contraction with respect to $\Phi$, we have
\begin{align*}
|\Phi(\xi_n,\xi_{n+1})|\leq c\left(|\Phi(\xi_{n-1},\xi_n)|+\left||\Phi(\xi_{n-1},\xi_n)|-|\Phi(\xi_{n+1},\xi_n)|\right|\right)
\end{align*}
for all $n\in\mathbb{N}$ and $c\in (0,1)$. Let us suppose that $|\Phi(\xi_{n-1},\xi_n)|\leq |\Phi(\xi_{n+1},\xi_n)|$ for some $n\in\mathbb{N}$. Then
\begin{align*}
|\Phi(\xi_n,\xi_{n+1})|&\leq c\left(|\Phi(\xi_{n-1},\xi_n)|+|\Phi(\xi_{n+1},\xi_n)|-\Phi(\xi_{n-1},\xi_n)|\right)\\
&\leq c |\Phi(\xi_{n+1},\xi_n)|,
\end{align*} 
which is a contradiction. Thus the only possibility is $|\Phi(\xi_n,\xi_{n+1})|\leq |\Phi(\xi_{n-1},\xi_{n})|$ for some $n\in\mathbb{N}$. Then
\begin{align*}
|\Phi(\xi_n,\xi_{n+1})|&\leq c\left(|\Phi(\xi_{n-1},\xi_n)|+|\Phi(\xi_{n-1},\xi_n)|-|\Phi(\xi_{n+1},\xi_n)|\right)\\
&\leq c\left(2|\Phi(\xi_{n-1},\xi_n)|-|\Phi(\xi_{n+1},\xi_n)|\right)\\
(1+c)|\Phi(\xi_n,\xi_{n+1})|& \leq c\left(2|\Phi(\xi_{n-1},\xi_n)|\right)\\
|\Phi(\xi_n,\xi_{n+1})|& \leq \frac{2c}{1+c}\left|\Phi(\xi_{n-1},\xi_n)\right|\\
&\leq \left(\frac{2c}{1+c}\right)^2\left|\Phi(\xi_{n-2},\xi_{n-1})\right|\\
& \vdots\\
&\leq \left(\frac{2c}{1+c}\right)^n \left|\Phi(\xi_0,\xi_1)\right|.
\end{align*}
Suppose that $m>n$ and $n\in\mathbb{N}$. Let $m=n+p$ where $p\geq 1$. Then by the given condition and the above inequality, we have
\begin{align*}
|\Phi(\xi_n,\xi_{n+p})|&\leq |\Phi(\xi_n, \xi_{n+1})|+|\Phi(\xi_{n+1}, \xi_{n+2})|+|\Phi(\xi_{n+1}, \xi_{n+2})|+\dots+|\Phi(\xi_{n+p-1}, \xi_{n+p})|\\
& \leq \left(\frac{2c}{1+c}\right)^n |\Phi(\xi_0,\xi_1)|+ \left(\frac{2c}{1+c}\right)^{n+1} |\Phi(\xi_0,\xi_1)|\\&+\left(\frac{2c}{1+c}\right)^{n+2} |\Phi(\xi_0,\xi_1)|+\dots+\left(\frac{2c}{1+c}\right)^{n+p-1} |\Phi(\xi_0,\xi_1)|\\
&= \left\{1+\frac{2c}{1+c}+\left(\frac{2c}{1+c}\right)^2+\dots+\left(\frac{2c}{1+c}\right)^{p-1}\right \}\left(\frac{2c}{1+c}\right)^n|\Phi(\xi_0,\xi_1)|\\
&=\left(\frac{2c}{1+c}\right)^n \frac{1-\left(\frac{2c}{1+c}\right)^p}{1-\frac{2c}{1+c}}|\Phi(\xi_0,\xi_1)|.
\end{align*}
Taking limit on both sides on the above equation, we get
$\displaystyle{\lim_{n\to\infty}}|\Phi(\xi_n, \xi_{n+p})|\to 0$ as $n\to\infty$. Therefore, the sequence $(\xi_n)$ is a $\Phi$-Cauchy sequence. Since $\Xi$ is $\Phi$-complete and $\Re$ is $\Phi$-closed subset of $\Xi$, there exists a point $\xi^*\in \Re$ such that $|\Phi(\xi_n,\xi^*)|\to 0$ for all $n\in\mathbb{N}$. Again from \eqref{equa3} and given condition, we get
\begin{align*}
|\Phi(\xi^*, \mathcal{F}(\xi_n))|&\leq |\Phi(\xi^*, \xi_{n+1})|+|\Phi(\xi_{n+1}, \mathcal{F}(\xi_n))|\\
&\leq |\Phi(\xi^*, \xi_{n+1})|+ D_\Phi(\Re, \Omega)\\
&\leq |\Phi(\xi^*, \xi_{n+1})|+ D_\Phi(\xi^*, \Omega).
\end{align*}
Taking limit on both sides, we have $|\Phi(\xi^*,\mathcal{F}(\xi_n))|\to|\Phi(\xi^*, \Omega)|$ as $n\to \infty$. Since $\Omega$ is approximately $\Phi$-compact with respect to $\Re$, there exists a subsequence $\left(\mathcal{F}(\xi_{n_k})\right)$ of $\left(\mathcal{F}(\xi_n\right))$ such that $\left|\Phi\left(\mathcal{F}(\xi_{n_k}),\gamma\right)\right|\to 0$ as $k\to \infty$ for some $\gamma\in\Omega$. Then
\begin{align*}
|\Phi(\xi^*,\gamma)|&\leq \lim_{n\to\infty}\left[|\Phi(\xi^*, \xi_{n_{k+1}})|+|\Phi( \xi_{n_{k+1}}, \mathcal{F}(\xi_{n_k}))|+|\Phi( \mathcal{F}(\xi_{n_k}), \gamma)|\right]\\
&=\lim_{n\to\infty}|\Phi( \xi_{n_{k+1}}, \mathcal{F}(\xi_{n_k}))|= D_\Phi(\Re, \Omega).
\end{align*}
Therefore, $|\Phi(\xi^*,\gamma)|=D_\Phi(\Re, \Omega)$ implies $\gamma\in \Re_\Phi$. Also since $\mathcal{F}(\Re_\Phi)\subseteq \Omega_\Phi$, there exists an element $\lambda\in \Re_\Phi$ such that 
\begin{align}\label{equa4}
|\Phi(\lambda, \mathcal{F}(\xi^*))|=D_\Phi(\Re,\Omega).
\end{align}
Now we can assume without loss of generality that $\xi^*\neq\xi_{n}$ for all $n\in \mathbb{N}$. Then from \eqref{equa3}, \eqref{equa4} and the definition of topological $p$-proximal contraction, we get
\begin{align*}
|\Phi(\xi_{n+1}, \lambda)|\leq c\left(|\Phi(\xi_n, \xi^*)|+\left||\Phi(\xi_{n+1}, \xi_n)|-|\Phi(\lambda, \xi^*)|\right|\right)
\end{align*}
for all $n\in \mathbb{N}$. Taking limit both sides, we obtain
\begin{align*}
|\Phi(\xi^*, \lambda)|\leq c |\Phi(\lambda, \xi^*)|,
\end{align*}
which gives that $\xi^*=\lambda$. Therefore from \eqref{equa4} we conclude that $\xi^*$ is the best proximity point of the mapping $\mathcal{F}$. For uniqueness, let us assume that there are two different best proximity points $\xi^*$ and $\lambda^*$ of the mapping $\mathcal{F}$ on $\Re$. Then we get
\begin{equation*}
        \begin{array}{ll}
            |\Phi(\xi^*,\mathcal{F}(\xi^*))|=D_\Phi(\Re, \Omega)\\
            |\Phi( \lambda^*, \mathcal{F}(\lambda^*))|=D_\Phi(\Re, \Omega).
        \end{array}
\end{equation*}
Since $\mathcal{F}$ is a topological $p$-proximal contraction, then we have
\begin{align*}
|\Phi(\xi^*, \lambda^*)|&\leq c\left(|\Phi(\xi^*, \lambda^*)|+\left||\Phi(\xi^*,\xi^*)|-|\Phi(\lambda^*, \lambda^*)|\right|\right)\\
&=c|\Phi(\xi^*, \lambda^*)|,
\end{align*}
which is a contradiction. Therefore $\lambda^*=\xi^*$. This completes the proof.
\end{proof}
Next we illustrate an example in support of the previously discussed Theorem \ref{thm1}.
\begin{example}
Consider $\mathbb{R}^2$ with usual topology and $\Xi=\{0\}\times [-1, 1]$ with subspace topology. Let $\Re=\left\{\left(0,t\right):-1\leq t\leq 0\right\}$ and $\Omega=\left\{\left(0,s\right):0\leq s\leq 1\right\}$. We define the mapping $\mathcal{F}:\Re\to\Omega$ as
\[\mathcal{F}(x,y)=\begin{array}{ll}
\left(0,\frac{y}{4}\right), x\in \Re.
\end{array}
\]
Let $\Phi:\Xi\times\Xi\to \mathbb{R}$ defined by
$$\Phi\left((\xi_1,\xi_2),(\lambda_1,\lambda_2)\right)=\lambda_2-\xi_2.$$
Then $\Phi$ is a continuous mapping on $\Xi\times \Xi$ and $D_\Phi(\Re, \Omega)=0.$ Now we show that $\Xi$ is $\Phi$-complete. Let $((0, \xi_n ))$ be a $\Phi$-Cauchy sequence in $\Xi$. Then $|\Phi((0,\xi_n), (0,\xi_{n+p}))|\to 0$ as $n\to \infty~~\mbox{and}~~p=1,2,3,\dots$ which implies $|\xi_n-\xi_{n+p}|\to 0$ as $n\to \infty~~\mbox{and}~~p=1,2,3,\dots$ So the sequence $(\xi_n)$ is a Cauchy sequence in $[0, 1]$. Since $[-1, 1]$ is complete, so there exists $\xi\in [-1, 1]$ such that $\xi_n\to \xi$ as $n\to \infty$. Then $|\Phi\left((0,\xi_n), (0, \xi)\right)|=|\xi_n-\xi|\to 0$ as $n\to \infty$. This shows that $((1,\xi_n))$ is $\Phi$-convergent to $(0, \xi)\in \Xi$. Therefore $\Xi$ is $\Phi$-complete. Now, let $(0,\xi)\in \Re_\Phi$, then there exists $(0, \lambda)\in \Omega$ such that $\left|\Phi\left((0, \xi),(0, \lambda)\right)\right|=0$ implies $|\xi-\lambda|=0$. This is satisfied only when $\xi=\lambda=0$. Therefore $\Re_\Phi=\{(0,0)\}$ and also $\Omega_\Phi=\{(0,0)\}$. Thus $\Re_\Phi$ is non-empty, $\Phi$-closed and also $\mathcal{F}(\Re_\Phi)\subseteq \Omega_\Phi$. Let $(0,\xi_1), (0,\xi_2),(0,\lambda_1), (0,\lambda_2)\in \Re$ with $(0,\lambda_1)\neq (0,\lambda_2)$ such that
\[|\Phi((0,\xi_1),\mathcal{F}(0,\lambda_1)|=D_\Phi(\Re, \Omega)=0\]
and
\[|\Phi( (0,\xi_2), \mathcal{F}(0,\lambda_2))|=D_\Phi(\Re, \Omega)=0.\]
That is $|\xi_1-\frac{\lambda_1}{4}|=0$ and $|\xi_2-\frac{\lambda_2}{4}|=0$. Now,
\begin{align*}
\left|\Phi\left((0,\xi_2),(0,\xi_2)\right)\right|&=|\xi_1-\xi_2|=\frac{1}{4}|\lambda_1-\lambda_2|\\
& \leq \frac{1}{4}|\lambda_1-\lambda_2|+\frac{1}{4}\\
&\leq \frac{1}{4}\left(|\lambda_1-\lambda_2|+\left||\xi_1-\lambda_1|-|\xi_2-\lambda_2|\right|\right)\\
&= c\left(|\Phi(\lambda_1,\lambda_2)|+\left||\Phi(\xi_1,\lambda_1)|-|\Phi(\xi_2,\lambda_2)|\right|\right),
\end{align*}
for $c=\frac{1}{4}$. This assures that the mapping $\mathcal{F}$ is a topological $p$-proximal contraction with respect to $\Phi$. Also all the hypotheses of the Theorem \ref{thm1} are satisfied. Therefore, we can conclude that the mapping $\mathcal{F}$ has a unique best proximity point $(0,0)$. 
\end{example}
In the subsequent discussion, we first coin the concept of topological $p$-contractive mappings and further, confirm a result of best proximity point theorem related to topological $p$-proximal contractive mappings on topological spaces.
\begin{definition}\label{def2}
Let $\Re, \Omega$ be two non-empty subsets of a topological space $\Xi$. A mapping $\mathcal{F}:\Re\to\Omega$ is  said to be a topological $p$-proximal contractive mapping if 
\begin{equation}\label{equa5}
\left.
        \begin{array}{ll}
            |\Phi(\alpha_1,\mathcal{F}( \beta_1))|=D_\Phi(\Re, \Omega)\\
            |\Phi(\alpha_2, \mathcal{F}(\beta_2))|=D_\Phi(\Re, \Omega)
        \end{array}
    \right\}\Rightarrow |\Phi(\alpha_1, \alpha_2)|< |\Phi(\beta_1, \beta_2)|+||\Phi(\alpha_1, \beta_1)|-|\Phi(\alpha_2,\beta_2)||
\end{equation}
for all $\alpha_1, \alpha_2,\beta_1, \beta_2\in \Re$ with $\beta_1\neq\beta_2$.
\end{definition}
Note that if $\Xi$ be a metric space and $\Phi$ be a metric  on $\Xi$, then the Definition \ref{def2} reduces to the notion of usual $p$-proximal contractive mappings. Also if $\Re=\Omega=\Xi$, then the inequality \eqref{equa5} can be written as
\begin{equation*}
|\Phi(\mathcal{F}(\beta_1),\mathcal{F}(\beta_2))|<|\Phi(\beta_1,\beta_2)|+||\Phi(\beta_1, \mathcal{F}(\beta_2))|-|\Phi(\beta_2,\mathcal{F}(\beta_2))||
\end{equation*}
for all $\beta_1, \beta_2\in \Xi$. 

Now, we are in a position to secure a best proximity point theorem for topological $p$-proximal contractive mappings. Further, we give an example to validate our finding.
\begin{theorem}\label{thm2}
Let $\Xi$ be a  topological space, where $\Phi:\Xi\times\Xi\to \mathbb{R}$ is a continuous function satisfying the following
$|\Phi(x,y)|=|\Phi(y,x)|$ and
 $|\Phi(x,z)|\leq |\Phi(x,y)|+|\Phi(y,z)| $ \mbox{for all} $x,y,z\in \Xi$. Let $\Re,\Omega$ be non-empty subsets of $\Xi$ and $\mathcal{F}:\Re\to\Omega$ be a topological $p$-proximal contractive mapping with respect to $\Phi$. Assume that $(\Re,\Omega)$ has the topological $p$-property with respect to $\Phi$ and $\mathcal{F}(\Re_\Phi)\subseteq \Omega_\Phi$. If there exists $\xi, \gamma\in\Re_\Phi$ such that
\begin{align}\label{equa6}
|\Phi(\xi, \mathcal{F}(\lambda))|=D_\Phi(\Re, \Omega)
\end{align} and
\begin{align}\label{equa7}
|\Phi(\xi, \lambda)|\leq |\Phi(\mathcal{F}(\xi),\mathcal{F}( \lambda))|.
\end{align}
Then $\mathcal{F}$ has a unique best proximity point.
\end{theorem}
\begin{proof}
Let $\xi, \lambda\in \Re_\Phi$. Since $\mathcal{F}(\Re_\Phi)\subseteq \Omega_\Phi$, then there exists a point $\gamma\in\Re_\Phi$ such that
\begin{align}\label{equa8}
|\Phi(\gamma, \mathcal{F}(\xi))|=D_\Phi(\Re,\Omega).
\end{align}
Employing the fact that $(\Re, \Omega)$ has the topological $p$-property and using \eqref{equa6}-\eqref{equa8}, we have
\begin{align}\label{equa9}
|\Phi(\xi, \gamma)|=|\Phi\left(\mathcal{F}(\lambda), \mathcal{F}(\xi)\right)|.
\end{align}
Let us assume $\xi\neq \lambda$. As $\mathcal{F}$ is a topological $p$-proximal contractive mapping and applying \eqref{equa6}-\eqref{equa8}, we have
\begin{align}\label{equa10}
|\Phi(\xi,\gamma)|<|\Phi(\xi, \lambda)|+\left| |\Phi(\xi,\lambda)|- |\Phi(\gamma, \xi)|\right|.
\end{align}
Then from \eqref{equa10}, by using \eqref{equa7} and \eqref{equa9}, we get
\begin{align*}
|\Phi(\xi,\gamma)|<|\Phi(\xi, \lambda)|+ |\Phi(\gamma, \xi)|-|\Phi(\xi,\lambda)|=|\Phi(\gamma, \xi)|,
\end{align*}
which is a contradiction. Therefore $\xi=\lambda$. That is $\mathcal{F}$ has a best proximity point.
For uniqueness, let us assume that there are two different best proximity points $\xi^*$ and $\lambda^*$ of the mapping $\mathcal{F}$ on $\Re$. Then we get
\begin{equation*}
        \begin{array}{ll}
            |\Phi(\xi^*,\mathcal{F}(\xi^*)|=D_\Phi(\Re, \Omega)\\
            |\Phi( \lambda^*, \mathcal{F}(\lambda^*))|=D_\Phi(\Re, \Omega).
        \end{array}
\end{equation*}
Since $\mathcal{F}$ is a topological $p$-proximal contractive mapping, we have
\begin{align*}
|\Phi(\xi^*, \lambda^*)|&< |\Phi(\xi^*, \lambda^*)|+\left||\Phi(\xi^*,\xi^*|-|\Phi( \lambda^*, \lambda^*|\right|=|\Phi(\xi^*, \lambda^*)|,
\end{align*}
which is a contradiction. Therefore $\lambda^*=\xi^*$. Therefore $\mathcal{F}$ has a unique best proximity point.
\end{proof}
Here we present one corollary of our obtained theorem which is the generalization of Edelstein fixed point theorem \cite{E1} on metric spaces.
\begin{corollary}
If we consider, $|\Phi(\beta_1, \mathcal{F}(\beta_2))|=|\Phi(\beta_2,\mathcal{F}(\beta_2))|$ in Theorem \ref{thm2}, then we get the Edelstein fixed point theorem in the topological space $\Xi$.
\end{corollary}
Next we furnish by giving a supporting example of the above Theorem.
\begin{example}
Consider $\mathbb{R}^2$ with usual topology and $\Xi=[-1, 1]\times [-1, 1]$ with subspace topology. Let $\Re=\left\{ 0\right\}\times[-1,0]$ and $\Omega=\left\{ 1\right\}\times [0, 1]$. 
Let $\Phi:\Xi\times\Xi\to \mathbb{R}$ defined by
$$\Phi\left((\xi_1,\xi_2),(\lambda_1,\lambda_2)\right)=\xi_2^2-\lambda_2^2.$$
Then $\Phi$ is a continuous mapping on $\Xi\times \Xi$ and $D_\Phi(\Re, \Omega)=0.$ We define the mapping $\mathcal{F}:\Re\to\Omega$ as
\[\mathcal{F}(0,t)=\begin{array}{ll}
\left(1,-\frac{t}{5}\right),
\end{array}
\]
for all $t\in [-1,0].$ Now it can be easily verified that $\Xi$ is $\Phi$-complete. Now, let $(0,\xi)\in \Re_\Phi$. Then there exists $(0, \lambda)\in \Omega$ such that $\left|\Phi\left((0, \xi),(0, \lambda)\right)\right|=0$ implies $|\xi^2-\lambda^2|=0$. This is satisfied only when $\xi=\lambda=0$. Therefore $\Re_\Phi=\{(0,0)\}$ and also $\Omega_\Phi=\{(0,0)\}$. Thus $\Re_\Phi$ is non-empty, $\Phi$-closed and also $\mathcal{F}(\Re_\Phi)\subseteq \Omega_\Phi$. Then the mapping $\mathcal{F}$ is topological $p$-proximal contractive with respect to $\Phi$ and $(\Re, \Omega)$ satisfies the $p$-property. Additionally for $\xi =\lambda =(0,0)$, the conditions \eqref{equa6} and \eqref{equa7} hold and further, all the hypotheses of Theorem \ref{thm2} are satisfied. Therefore, we can conclude that the mapping $\mathcal{F}$ has a unique best proximity point. 
\end{example}
\begin{Acknowledgement}
The Research is funded by the Ministry of Education, Government of India. 
\end{Acknowledgement}
\bibliographystyle{plain}

\begin{thebibliography}{10}
\bibitem{AAS}
I. Altun, M. Aslantas, and H. Sahin.
\newblock Best proximity point results for $P$-proximal contractions. 
\newblock {\em Acta Math. Hungar.}, 16(2):393--402, 2020.

\bibitem{B1}
S.~Banach.
\newblock Sur les op\'erations dans les ensembles abstraits et leur application
  aux \'equations int\'egrales.
\newblock {\em Fund. Math.}, 3:133--181, 1922.

\bibitem{Bss}
S.S.~Basha.
\newblock Best proximity points: optimal solutions.
\newblock {\em J. Optim. Theory Appl.}, 151(1):210--216, 2011.

\bibitem{bw}
D.W.Boyd, J.S.W. Wong.
\newblock On nonlinear contractions.
\newblock {\em Proc. Amer. Math. Soc.},  20:458--464, 1969.

\bibitem{DD}
P.~Das and L.K. Dey.
\newblock Fixed point of contractive mappings in generalized metric spaces.
\newblock {\em Math. Slovaca}, 59(4):499--504, 2009.

\bibitem{DM}
L.K. Dey and S.~Mondal.
\newblock Best proximity point of {F}-contraction in complete metric spaces.
\newblock {\em Bull. Allahabad Math. Soc.}, 30(2):173--189, 2015.

\bibitem{E1}
M.~Edelstein.
\newblock On fixed and periodic points under contractive mappings.
\newblock {\em J. Lond. Math. Soc.}, 37(1):74--79, 1962.

\bibitem{E5}
E. Karapınar.
\newblock  Best proximity points of cyclic mappings.
\newblock {\em Appl. Math. Lett.}, 25(11):1761-1766.

\bibitem{EV}
A.A.~Eldred and P.~Veeramani.
\newblock  Existence and convergence of best proximity points. 
\newblock {\em J. Math. Anal. Appl.} 323(2):1001--1006, 2006.

\bibitem{Gb}
M.~ Gabeleh.
\newblock Best proximity point theorems via non-self mapping.
\newblock {\em J. Optim. Theory Appl.}, 38(1):143--154, 2015.

\bibitem{HS}
P.~Hitzler and A.K. Seda.
\newblock Dislocated topologies.
\newblock {\em J. Electr. Eng.}, 51(12):3--7, 2000.

\bibitem{HZ}
L.G. Huang and X. Zhang.
\newblock Cone metric spaces and fixed point theorems of contractive mappings.
\newblock {\em J. Math. Anal. Appl.}, 332(2):1468--1476, 2007.

\bibitem{KDKC}
S. Karmakar, L.K. Dey, P. Kumam, A. Chanda.
\newblock Best proximity results for cyclic $\alpha$-implicit contractions in quasi-metric spaces and its consequences.
\newblock {\em Advances in Fixed Point Theory}, 7(3):342-358, 2017.

\bibitem{KRV}
W.A. Kirk, S. Reich and P. Veeramani.
\newblock Proximinal retracts and best proximity pair theorems
\newblock {\em Numer. Funct. Anal. Optim.}, 24(7-8):851-862, 2003.

\bibitem{SLDH}
S. Som, S. Laha, A. Petrusel and L.K. Dey.
\newblock Best proximity point results on arbitrary topological spaces and the Banach contraction
principle revisited.
\newblock {\em J. Nonlinear Convex Anal.}, 2021.

\bibitem{P}
O. Popescu.
\newblock A new type of contractive mappings in complete metric spaces.
\newblock Preprint.

\bibitem{R}
V. Sankar Raj.
\newblock A best proximity point theorem for weakly contractive non-self mappings.
\newblock {\em Nonlinear Anal.}, 74(14):4804-4808, 2011.

\bibitem{RP}
V. Sankar Raj and T. Piramatchi.
\newblock Best proximity point theorems is topological spaces.
\newblock {\em J. Fixed Point Theory Appl.}, 22(1), 2020. Article 2.

\bibitem{SBJ}
N. Shahzad, S.S. Basha, and R. Jeyaraj.
\newblock Common best proximity points: global optimal solutions.
\newblock {\em J. Optim. Theory Appl.}, 148(1):69--78, 2011.

\bibitem{SKV}
T.~ Suzuki, M.~ Kikkawa, and C.~ Vetro.
\newblock The existence of best proximity points in metric spaces
with the property UC.
\newblock {\em Nonlinear Anal., Theory Methods Appl.}, 71(7):2918--2926, 2009.


\bibitem{SSJ}
S.S. Basha, N.~Shahzad and R.~Jeyaraj.
\newblock Best proximity point theorems: exposition of a significant
non-linear programming problem.
\newblock {\em J. Glob. Optim.}, 56(4):1699--1705, 2013.

\bibitem{V}
C.~Vetro.
\newblock Best proximity points: convergence and existence theorems for $p$-cyclic mappings.
\newblock {\em  Nonlinear Anal.}, 73(7): 2283--2291, 2010.

\bibitem{WPB}
K.~W\l odarczyk, R.~Plebaniak, and A.~Banach.
\newblock  Best proximity points for cyclic and noncyclic set-valued
relatively quasi-asymptotic contractions in uniform spaces. 
\newblock {\em Nonlinear Anal.}, 70(9):3332--3341, 2009.
\end{thebibliography}

\end{document}